\documentclass[letterpaper]{amsart}
\usepackage{amssymb}
\usepackage{pgf}
\usepackage{tikz}
\usetikzlibrary{arrows,automata}

\usepackage[hidelinks,pagebackref]{hyperref}

\renewcommand*{\backref}[1]{}
\renewcommand*{\backrefalt}[4]{
  \ifcase #1 
  [No citations.]
  \or [#2]
  \else [#2]
  \fi }

\newtheorem{theorem}{Theorem}
\newtheorem{lemma}[theorem]{Lemma}

\newtheorem{corollary}[theorem]{Corollary}

\theoremstyle{definition}
\newtheorem{definition}[theorem]{Definition}

\newtheorem{remark}[theorem]{Remark}

\newtheorem{question}[theorem]{Question}
\newtheorem{conjecture}[theorem]{Conjecture}

\newcommand{\st}{\mathbin{\mid}} %%% \mathbin = binary operator,
\newcommand{\from}{\colon}
\newcommand{\semi}{\rtimes}
\newcommand{\isom}{\cong} % Isomorphism
\newcommand{\cross}{\times}

\newcommand{\CC}{\mathbb{C}}
\newcommand{\EE}{\mathbb{E}}
\newcommand{\HH}{\mathbb{H}}
\newcommand{\NN}{\mathbb{N}}
\newcommand{\RR}{\mathbb{R}}
\newcommand{\ZZ}{\mathbb{Z}}

\newcommand{\WP}{\operatorname{WP}}
\newcommand{\BS}{\operatorname{BS}}
\newcommand{\MCF}{\operatorname{MCF}}
\newcommand{\PSL}{\operatorname{PSL}}
\newcommand{\Sol}{\operatorname{Sol}}
\newcommand{\Nil}{\operatorname{Nil}}

\newcommand{\group}[2]{\langle #1 \st #2 \rangle}
\newcommand{\subgp}[1]{{\langle #1 \rangle}}

\newcommand{\refsec}[1]{Section~\ref{Sec:#1}}
\newcommand{\refthm}[1]{Theorem~\ref{Thm:#1}}
\newcommand{\reflem}[1]{Lemma~\ref{Lem:#1}}
\newcommand{\refrem}[1]{Remark~\ref{Rem:#1}}

\newcommand{\fakeenv}{} %%% prints the emptystring

\newenvironment{restate}[2]  %%% restate takes two arguments 
{ 
 \renewcommand{\fakeenv}{#2} %%% So now \fakeenv prints #2
 \theoremstyle{plain} 
 \newtheorem*{\fakeenv}{#1~\ref{#2}} %%% so now #2 is the name of a
 \begin{\fakeenv}
}
{
 \end{\fakeenv}
}

\title{Groups whose word problems are not semilinear}

\author[Gilman]{Robert H. Gilman}
\email{rgilman@stevens.edu}

\author[Kropholler]{Robert P. Kropholler}
\email{robert.kropholler@gmail.com}

\author[Schleimer]{Saul Schleimer}
\email{s.schleimer@warwick.ac.uk}

\thanks{This material is based upon work supported by the National
  Science Foundation grant DMS-1440140 while the authors were in
  residence at the Mathematical Science Research Institute (MSRI) in
  Berkeley, California, during the Fall 2016 Semester.}

\date{\today}

\begin{document}

\begin{abstract} 
Suppose that $G$ is a fintely generated group and $\WP(G)$ is the formal language of words defining the identity in $G$.
%%% Linguistic properties $WP(G)$ are connected with geometric
%%% properties of the Cayley diagram of $G$.
We prove that if $G$ is a nilpotent group, the fundamental group of a
finite volume hyperbolic three-manifold, or a right-angled Artin group whose graph lies in a certain infinite class, 
then $\WP(G)$ is not a multiple context free language.
%%% SS: reordered to reflect order in paper.  SS: Can we get rid of
%%% the word ``usually'' by being more precise?  I'd like to avoid
%%% ``weasel words'' in the abstract.
\end{abstract}

\maketitle

\section{Introduction}

The word problem for a finitely generated group $G$ is to decide if a
given word in the generators and their formal inverses, defines the
identity in $G$ or not.  This problem was proposed for finitely
presented groups by M.~Dehn~\cite{MR1511645} in 1911 and has been
profitably studied since then.  In 1971
A.~V.~Anisimov~\cite{MR0301981} introduced the word problem as a
formal language. The validity of this point of view was
confirmed by Muller and Schupp's result~\cite{MR710250} that the word
problem of $G$ is a context-free language if and only if $G$ is
virtually free.

Muller and Schupp's result inspired many authors. See for
example~\cite{MR3200359, MR3394671, MR3048114, MR2142503, MR2470538,
  MR2787458, MR2323454, MR1914990, MR3338967}. One intriguing aspect
of their work is the connection it reveals between the logical
complexity of the word problem, considered as a formal language, and
geometric properties of the Cayley diagram.  Context-free languages are generated by context-free grammars and are accepted by pushdown
automata.  For word problems of groups these two conditions correspond directly to the geometric properties:
\begin{enumerate}
\item cycles in the Cayley diagram are triangulable by diagonals of
  uniformly bounded length, and
\item the Cayley diagram has finitely many end isomorphism types, 
\end{enumerate} 
respectively. 

A natural question is whether there is a group whose word problem is
not context free, but is in the larger class of indexed languages. In
particular, is the word problem of $\ZZ^2$ indexed? These
questions have been open for decades.  Indexed languages form level
two of the OI hierarchy of language classes, and
S.~Salvati~\cite{MR3354791} has recently shown that the word problem
of $\ZZ^2$ is a multiple context-free ($\MCF$) language and hence at
level three of that hierarchy.  In addition, as with Muller and Schupp's result,
Salvati's linguistic characterization of the word problem of $\ZZ^2$
is closely related the geometry of its Cayley diagram.

It is of interest, then, to investigate which other groups have $\MCF$ word problem and what geometric conditions their Cayley diagrams might satisfy. In this paper we use the facts that $\MCF$ languages  form a cone~\cite{MR1131066} and are semilinear~\cite{Vijay-Shanker} to show that a large swath of groups do not have $\MCF$ word problem. More precisely we prove the following theorems.

\begin{restate}{Theorem}{Thm:Nilpotent}
Let $\CC$ be a cone of semilinear languages. If the word problem of a
finitely generated virtually nilpotent group $G$ is in $\CC$, then $G$ is virtually abelian.
\end{restate}

Meng-Che Ho~\cite{turbo} has recently shown that the word problem of $\ZZ^n$
is $\MCF$ for all $n$. Hence by \reflem{Subgroups} all finitely
generated virtually abelian groups have $\MCF$ word problems. We have the following corollary to \refthm{Nilpotent}.

\begin{corollary}
A finitely generated virtually nilpotent group has $\MCF$ word problem if and only if it is virtually abelian.
\end{corollary}

Our next theorem concerns fundamental groups of three-manifolds.

\begin{restate}{Theorem}{Thm:Hyperbolic}
Suppose that $M$ is a hyperbolic three-manifold.  Then $\WP(\pi_1(M))$ is not $\MCF$.
\end{restate}

Let $\mathcal{G}$ be the class of graphs containing a point and closed under the following operations:
\begin{itemize}
\item If $\Gamma,\Gamma'\in\mathcal{G}$, then $\Gamma\sqcup\Gamma'\in \mathcal{G}$, and
\item if $\Gamma\in \mathcal{G}$, then $\Gamma\ast\{v\}\in\mathcal{G}$
\end{itemize}
where $\sqcup$ denotes disjoint union and $\Gamma\ast\{v\}$ is the cone of $\Gamma$. It will be clear from the context whether we are
speaking of the cone of a graph or a cone of languages.

\begin{restate}{Theorem}{Thm:RAAG}
Let $\Gamma$ be a graph and $A(\Gamma)$ be the associated RAAG. If $A(\Gamma)$ has multiple context-free word problem, then $\Gamma\in \mathcal{G}$\end{restate}

These theorems are proved in
Sections~\ref{Sec:Nilpotent},~\ref{Sec:Hyperbolic} and~\ref{Sec:RAAG}
respectively. \refsec{bkgnd} contains relevant background material
including definitions of cones and semilinearity. For further
introduction to formal language theory see~\cite{MR0443446, MR526397,
  MR645539, MR1469992}. An introduction aimed at group theorists is
given in~\cite{MR1364178}. For properties of multiple context-free
languages consult~\cite{MR1131066} and~\cite{Kallmeyer2010}.

\section{Background}
\label{Sec:bkgnd}

\subsection{Formal languages}
Let $\Sigma$ be a finite \emph{alphabet}: that is, a nonempty finite set.  A
\emph{formal language} over $\Sigma$ is a subset of $\Sigma^*$, the
free monoid over $\Sigma$.  Elements of $\Sigma^*$ are called
\emph{words}.

A \emph{choice of generators} for a group $G$ is a surjective monoid
homomorphism $\pi \from \Sigma^* \to G$. We require that $\Sigma$ be
\emph{symmetric}: closed under a fixed-point-free involution
$\cdot^{-1}$.  We also require $\pi(a^{-1}) = \pi(a)^{-1}$ for all $a
\in \Sigma$.  The involution extends to all words over $\Sigma$ in the
usual way. 
Note that we adhere to the usual notation for group presentations. The choice of generators corresponding to a presentation 
$\langle a, t \mid tat^{-1}a^{−2}\rangle$ uses the alphabet $\Sigma=\{ a, a^{-1}, t, t^{-1}\}$ etc.

The \emph{word problem} for $G$ is the formal language $\WP(G) =
\pi^{-1}(1)$.  It is evident that $\WP(G)$ depends on the choice of
generators, but this dependence is mild.  As we will see below, whether
or not $\WP(G)$ is in any particular cone of formal languages is
independent of the choice of generators and depends only on $G$.

\subsection{Regular languages and finite automata}
A \emph{finite automaton} over $\Sigma$ is a finite directed graph with edges labelled
by words in $\Sigma^*$, a designated start vertex and a set of
designated accepting vertices. A word is \emph{accepted} by an automaton if
it is the concatenation of labels along a directed path from the start
vertex to an accepting vertex. The \emph{accepted language} is the set of all accepted words. The \emph{regular languages} over a finite alphabet $\Sigma$ are the
languages accepted by finite automata over $\Sigma$.

\begin{figure}[ht]
\begin{tikzpicture}[->,>=stealth',shorten >=1pt,auto,node distance=2.0cm, semithick]
  \node[state]            (A)                   {$q_a$};
  \node[state]            (B) [right of=A]      {$q_b$};
  \node[state]            (D) [above left of=B] {$q_d$};
  \node[accepting,state]  (C) [right of=B]      {$q_c$};
  \node[node distance= 1cm] (S) [left of=A]       {};
  \path (A) edge              node {$b$} (B)
            edge              node {$b$} (D)
        (B) edge [loop above] node {$c$} (B)
            edge              node {$b$} (C)
        (D) edge              node {$a$} (B)
        (S) edge                         (A);
\end{tikzpicture}
\caption{A finite automaton accepting the language $bc^*b+bac^*b$ \label{fig:automaton}} 
\end{figure}
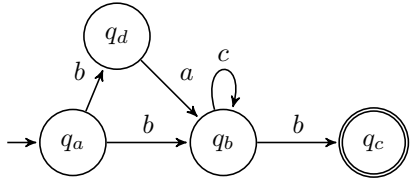

Figure~\ref{fig:automaton} shows a finite automaton with start vertex
$q_a$ and one accepting vertex, $q_c$. The regular language accepted
by this automaton may be denoted symbolically via the regular
expression $bc^*b+bac^*b$. Here $+$ stands for union and $^*$ for
submonoid closure.

\subsection{Transducers}
A \emph{transducer} $\tau$ (more precisely a rational transducer) is a finite automaton whose edge labels are pairs of words $(w,v)$ over finite alphabets $\Sigma, \Delta$ respectively. Path labels are obtained by concatenating the edge labels in each coordinate. The labels of all accepted p.aths form a subset of $\Sigma^*\times \Delta^*$. The image under $\tau$ of a language $L\subset\Sigma^*$ is $\tau(L) = \{ v \mid \mbox{there is some $w\in L$ with $(w,v)\in \tau$}\}$. 

\subsection{Cones}
A class, $\CC$, of languages is a \emph{cone} (also called a full trio~\cite[pages 201-202]{MR1469992}) if it
contains at least one nonempty language and is closed under the
following operations.
\begin{enumerate}
\item If $L\subset\Sigma_1^*$ is in $\CC$, and $\sigma \from
  \Sigma_1^*\to \Sigma_2^*$ is a monoid homomorphism, then $\sigma(L)$
  is in $\CC$.
\item If $L\subset\Sigma_2^*$ is in $\CC$, and $\sigma \from
  \Sigma_1^*\to \Sigma_2^*$ is a monoid homomorphism, then
  $\sigma^{-1}(L)$ is in $\CC$.
\item If $L\subset\Sigma_1^*$ is in $\CC$, and $R\subset
  \Sigma_1^*$ is regular, then $L \cap R$ is in $\CC$.
\end{enumerate}
In other words cones are closed under homomorphism, inverse
homomorphism and intersection with regular languages. The condition
on nonempty languages above is included to rule out the empty cone and
the cone consisting of the empty language.  Multiple context-free
languages form a cone~\cite{MR1131066}.

\begin{theorem}[Nivat's Theorem~\cite{MR0238633}]\label{Th:Nivat} If $L$ is in a cone $\tau$, then so is $\tau(L)$. In other words, cones are closed under transduction. 
\end{theorem}

As the following results are well known, we provide only sketches of the proofs.

\begin{lemma}
\label{Lem:Subgroups}
Let $\WP(G)$ be the word problem of $G$ with respect to a
choice of generators $\pi:\Sigma^*\to G$. Suppose $\WP(G)$ is in a cone of $\CC$ of formal languages.  Then:
\begin{enumerate}
\item The word problem for $G$ with respect to any choice of
  generators is in $\CC$.
\item The word problem for every finitely generated subgroup of $G$ is
  in $\CC$.
\item The word problem for every finite index supergroup of $G$ is
  in $\CC$.
\end{enumerate}
\end{lemma}

\begin{proof}
Suppose $\delta:\Delta^*\to G$ is any choice of generators for $G$ or one of its finitely generated subgroups. Since $\Delta^*$ is a free monoid, $\delta$ factors as $\pi\circ f$ for some monoid homomorphism $f:\Delta^*\to\Sigma^*$. It follows that $\delta^{-1}(1) = f^{-1}(\WP(G))\in \CC$.

Now suppose $G$ has finite index in a group $K$, and $\delta:\Delta^*\to K$ is a choice of generators. Since we are assuming that $\Delta$ is symmetric, we can partition it into a disjoint union $\Delta=\Delta_0\sqcup\Sigma_0^{-1}$. By Theorem~\ref{Th:Nivat} it suffices to show that $\WP(K)$ is the image of $\WP(G)$ under a transduction $\tau$. We define $\tau$ in three steps. 

First, recall that the vertices of the Schreier diagram, $\Gamma$, of $G$ in $H$ are the right cosets $\{Gx\}$ of $G$ in $H$, and that for each vertex $Gx$ and generator $a\in \Delta_0$ there is a directed edge labelled $a$ from $Gx$ to $Gxa$. Paths in $\Gamma$ may traverse edges in either direction, but an edge traversed aginst its orientation contributes the inverse of its label to label of a path. Fixing $G$ as the start vertex and sole accepting vertex makes $\Gamma$ into a finite automaton which accepts the regular language of all words over $\Delta$ which represent elements of $G$. 

Second, pick a spanning tree $\Gamma_0$ for $\Gamma$ with root $G$ and edges oriented in any direction. Each edge $e$ in $\Gamma - \Gamma_0$ determines a Schreier generator $uav^{-1}$ for $G$. Here $u$ is the label of the path in $\Gamma_0$ from $G$ to the source vertex of $e$, $v$ is the label of the path to the target vertex, and $a$ is the label of $e$.

Third make $\Gamma$ into a transducer by changing its labels into pairs of words. Existing edge labels become the second components of new edge labels. Each edge in the spanning tree has the empty word as the first component of its label, while ach edge $e$ not in the spanning tree has a new letter $b_e$ as the first component of its label.

Let $\Sigma$ be the alphabet of all the $b_e$'s and their formal inverses. The transducer $\Gamma$ defines a binary relation $\tau:\Sigma^*\to \Delta^*$. Define a monoid homomorphism $\pi:\Sigma^*\to G$ which sends each $b_e$ to the image under $\delta$ of its corresponding Schreier generator, and likewise for $b_e^{-1}$. It is straightforward to check first that $\pi(u)=\delta(v)$ for any $(u,v)\in \tau$ and second that $\tau(\WP(G)) = \WP(H)$.
\end{proof}

\subsection{Semilinearity}

For each $a_i\in \Sigma = \{a_1, \ldots, a_k\}$ and $w \in \Sigma^*$,
define $|w|_i$ to be the number of occurrences of $a_i$ in $w$.  The
\emph{Parikh map} $\psi \from \Sigma^* \to \NN^k$ sends $w$ to the vector
$(|w|_1,\ldots, |w|_k)$ where $\NN$ is the non-negative natural numbers.

A \emph{linear subset} of $\NN^k$ is one of the form $v_0 + \langle
v_1,\ldots v_m \rangle$, i.e, a translate of a finitely generated
submonoid. A \emph{semilinear subset} of $\NN^k$ is a finite union of linear subsets. A \emph{semilinear language} $L\subset \Sigma^*$ is is a language whose image under the map  $\psi \from \Sigma^* \to \NN^k$ defined above is semilinear. Multiple context-free languages are semilinear by~\cite{Vijay-Shanker}.

Since semilinearity is preserved by monoid homomorphisms $\NN^k\to \NN^m$, our discussion yields the following useful result.

\begin{lemma}
\label{Lem:Useful}
Suppose that $L\subset \Sigma^*$ is semilinear, and $R \subset
\Sigma^*$ is regular. Then the projection of $\psi(W \cap R)$ onto any nonempty subset of coordinates is semilinear. \qed   
\end{lemma} 
For short we say that the projection of a regular slice of a semilinear language onto a nonempty subset of coordinates is semilinear. We call the composition of these projections with Parikh map as Parikh maps too.

\section{Nilpotent groups}\label{Sec:Nilpotent}

The goal of this section is to prove the following.

\begin{theorem}
\label{Thm:Nilpotent}
Let $\CC$ be a cone of semilinear languages. If the word problem of a
finitely generated virtually nilpotent group $G$ is in $\CC$, then $G$ is virtually abelian.
\end{theorem}

Assume $G$ is virtually nilpotent but not virtually abelian with word problem in a semilinear cone $\CC$. By \reflem{Subgroups} we may assume without loss of generality that $G$ is nilpotent; that is, $G$ has an ascending central series
\[
1=Z_0 \subset Z_1 \subset \cdots \subset Z_k = G
\]
where $Z_{i+1}/Z_i$ is the center of $G/Z_i$. If $k=1$, there is nothing to prove, so we assume $k\ge 2$.

Recall the notation for the commutator $[g,h] = g^{-1}h^{-1}gh$, and recall also that subgroups of a finitely generated nilpotent group are themselves finitely generated. We divide the rest of the proof into two lemmas. 

\begin{lemma}\label{Lem:ghz}
There exist $g\in G$, $h\in Z_2$ with $[g,h]$ of infinite order. 
\end{lemma}

\begin{proof}
Suppose for all choices of $g,h$ as above, $[g,h]$ has finite
order. Then every $[g,h]$ lies in the torsion subgroup of $Z_1$ whence
the orders of the $[g,h]$'s are uniformly bounded by some integer
$m$. It follows that $[g,h^m] = [g,h]^m = 1$ for all $g,h$. But then
$Z_2/Z_1$ is a finitely generated abelian torsion group and hence finite. By~\cite[Lemma 0.1]{MR0283082} a finitely generated nilpotent group with finite
center is finite. Thus $G/Z_1$ is finite and $Z_1$ is abelian of
finite index, which contradicts our assumption that $G$ is not virtually abelian.
\end{proof}

Without loss of generality $\Sigma$
contains letters $a_g, a_h, a_z$ which project to $g,h,z$ respectively. Let $W = \WP(G)$ be the word problem of $G$.
\[ W\cap a_g^*a_h^*(a_g^{-1})^*(a_h^{-1})a_z^* = \{a_g^ma_h^n(a_g^{-1})^m(a_h^{-1})^na_z^{mn}\}.\]
Since $W$ is semilinear by hypothesis, \reflem{Useful} implies that $S =\{(m, mn) \mid m,n \in \NN \}$
is semilinear. Thus the following lemma completes the proof of \refthm{Nilpotent}.

\begin{lemma}\label{Lem:iij}
$S = \{(m, mn) | m, n \in \NN\}$ is not semilinear.
\end{lemma}

\begin{proof}
Observe that if distinct elements of $S$ share the same first
coordinate, then their second coordinates differ by at least the size
of that first coordinate. It follows that $S$ does not contain a
linear subset of the form
\[
(p, q) + \langle (r, s), (0, t) \rangle
\]
with $r \neq 0 \neq t$. Indeed $S$ would then contain both $(p+kr,
q+ks)$ and $(p+kr, q+ks+t)$ for all integers $k> 0$ contrary to our
observation above.

Thus either all the module generators for any linear subset of $S$
have first coordinate $0$ or none do (as we may safely assume that
$(0,0)$ is not a generator). Modules of the first type are contained
in $\{0\}\times \NN$, and the slopes of elements (thought of as vectors
based at the origin) of a module of the second type are bounded above
by the maximum of the slopes of its generators.

We see that if $S$ were semilinear then the slopes of all elements
whose first coordinates are large enough would be uniformly bounded,
which is not the case.
\end{proof}

\section{Fundamental groups of hyperbolic three-manifolds}
\label{Sec:Hyperbolic}

\subsection{Distortion}
\label{Sec:Distortion}

We begin with a simple example that illustrates the main idea of this
section.  Suppose that $G = \BS(1, 2) = \group{a, t}{tat^{-1}a^{-2}}$
is a Baumslag-Solitar group~\cite{MR0142635}.  We claim that $W = \WP(G)$ is not
multiple context free ($\MCF$).  Consider the regular language $R =
t^* a (t^{-1})^* A^*$ and form the rational slice $W \cap R$.
Abelianizing tells us that in any word $w \in W \cap R$ the powers of
$t$ and $T$ appearing must be equal.  Thus we have $W \cap R = \{ t^n
a t^{-n} a^{-2^n} \st n \in \NN \}$.  We now apply the Parikh map
$\psi = (|\cdot|_t, |\cdot|_{a^{-1}})$.  The image $\psi(W \cap R)$ is
the graph of $f(n) = 2^n$, lying inside of $\NN^2$.  Clearly any line
meets the image in at most two points.  Thus $\psi(W \cap R)$ is not
semilinear and so $W$ is not $\MCF$ by \reflem{Useful}.

Suppose that $G$ is a group and $H$ is a subgroup.  Fix a generating
set $\Sigma$ for $G$ that contains a generating set $\Sigma_H$ for
$H$.  Let $\Gamma$ and $\Gamma_H$ be the corresponding Cayley graphs.
The inclusion of $H$ into $G$ gives a Lipschitz map $\Gamma_H \to
\Gamma$.  The failure of this map to be bi-Lipschitz measures the
\emph{distortion} of $H$ inside of $G$.  In the $\BS(1,2)$ example,
the distortion of the subgroup $H = \subgp{a}$ is exponentially large.

The general principle is as follows.  If $G$ has a distorted subgroup
$H$, and $H$ has a sufficiently ``regular'' sequence of elements, then
$\WP(G)$ is not $\MCF$.

\begin{question}
\label{Que:Distorted}
Suppose that $G$ has a subgroup $H$ with super-linear distortion.
Does this imply that $\WP(G)$ is not $\MCF$?
\end{question}
%%% This seems a bit tricky.  Is distortion always seen via a cyclic
%%% subgroup of $H$?  I'll guess not... But even if it is, how do we
%%% encode enough of the cyclic subgroup via a regular expression over
%%% $\calA$?

\subsection{Fundamental groups}
\label{Sec:Fundamental}

We say that a manifold $M$ is \emph{hyperbolic} if $M$ admits a
Riemannian metric, of constant sectional curvature minus one, which is
complete and which has finite volume.  Using deep results from
low-dimensional topology we will prove the following.

\begin{theorem}
\label{Thm:Hyperbolic}
Suppose that $M$ is a hyperbolic three-manifold.  Then $\WP(\pi_1(M))$
is not $\MCF$.
\end{theorem}

\noindent
Before giving the proof we provide the topological background.
Suppose that $S$ is a hyperbolic surface.  Suppose that $f \from S \to
S$ is a homeomorphism.  We form $M_f$, a \emph{surface bundle over the
  circle}, by taking $S \cross [0, 1]$ and identifying $S \cross
\{1\}$ with $S \cross \{0\}$ using the map $f$.  The gluing map $f$ is
called the \emph{monodromy} of the bundle.  The surface $S$ is called
the \emph{fiber} of the bundle; in a small abuse of notation $M_f$ is
also simply called a \emph{fibered} manifold.

Let $\phi \from \pi_1(S) \to \pi_1(S)$ be the homomorphism induced by
$f$.  Note that
\[
\pi_1(M_f) \isom \pi_1(S) \semi_{\phi} \ZZ
            = \group{\Sigma, t}{tat^{-1} = \phi(a),\, a \in \Sigma}
\]
where $\Sigma$ generates $\pi_1(S)$.

It is a result of Thurston~\cite[Theorem~5.6]{MR648524} that a
fibered manifold $M_f$ is hyperbolic if and only if the monodromy $f$
is \emph{pseudo-Anosov}.  Instead of giving the definition here, we
will simply note an important consequence~\cite[Theorem~5]{MR956596}:
If $f \from S \to S$ is pseudo-Anosov then, for any letter $a \in
\Sigma$, the word-lengths of the elements $\phi^n(a)$ grow
exponentially.

One sign of the importance of surface bundles to the theory of
three-manifolds is Thurston's virtual fibering
conjecture~\cite[Question~6.18]{MR648524}: every hyperbolic
three-manifold has a finite cover which is fibered.  This remarkable
conjecture is now a theorem, due to
Wise~\cite[Corollary~1.8]{MR2558631} in the non-compact case and due
to Agol~\cite[Theorem~9.2]{MR3104553} in the compact case.  (For a
detailed discussion, including many references, please
consult~\cite{MR3444187}.)  Note that any finite cover of a hyperbolic
manifold is again hyperbolic.  Thus, by Thurston's theorem, the
monodromy of the fibered finite cover is always pseudo-Anosov.

We are now ready for the proof. 

\begin{proof}[Proof of \refthm{Hyperbolic}]
Suppose that $M$ is a hyperbolic three-manifold.  Appealing to
\reflem{Subgroups} and to the solution of the virtual fibering conjecture we may replace $M$ by a fibered finite cover $M_f$, with fiber $S$.
Fix $\Sigma$ a generating set for $\pi_1(S)$ and let $t$ be the stable
letter, representing the action of the monodromy.  Thurston tells us
that $f$ is pseudo-Anosov, and thus for any generator $a \in \Sigma$
the elements $\phi^n(a)$ grow exponentially in the word metric on
$\pi_1(S)$.

So $G = \pi_1(M_f)$ is generated by $\Sigma \cup \{t\}$ and has the
presentation given above.  Set $W = \WP(G)$ and set $R = t^* a
(t^{-1})^* \Sigma^*$.  Homological considerations imply that
\[
W \cap R = \{t^n a t^{-n} w^{-1} \st n \in \NN, w \in \Sigma^*, w =_G \phi^n(a)\}.
\]
Define $|w|_\Sigma = \sum_{b \in \Sigma} |w|_b$ and consider the
Parikh map $\psi = (|\cdot|_t, |\cdot|_\Sigma)$.  The image $\psi(W
\cap R) \subset \NN^2$ contains, and lies above, the graph of an
exponentially growing function.  Thus its intersection with any non-vertical line is finite.  We deduce from \reflem{Useful} that $W$ is not $\MCF$.
\end{proof}

\begin{remark}
\label{Rem:OtherGeometries}
Five of the remaining seven Thurston geometries are easy to dispose
of.  In $S^3$ geometry, all fundamental groups are finite.  In $S^2
\cross \RR$ and in $\EE^3$ geometry all fundamental groups are
virtually abelian and so they are all MCF.   In $\Nil$ geometry all fundamental groups are
virtually nilpotent yet not virtually abelian.  Thus
\refthm{Nilpotent} applies; none of these fundamental groups are MCF.  In $\Sol$ geometry all manifolds are
finitely covered by a torus bundle with Anosov monodromy.  Thus the
discussion of this section applies and these groups do not have word
problem in $\MCF$.

The question is open for the geometries $\HH^2 \cross \RR$ and
$\PSL(2, \RR)$ geometry, for both uniform and non-uniform lattices.

%%% I get confused easily by PSL(2, \RR) geometry. 
\end{remark}

We end this section with another obvious question. 

\begin{question}
\label{Que:Surface}
Suppose that $S_g$ is the closed, connected, oriented surface of genus
$g > 1$.  Is the word problem for $\pi_1(S_g)$ multiple context free?
\end{question}

\section{Right-angled Artin groups}
\label{Sec:RAAG}

Let $\mathcal{G}$ be the class of graphs containing a point and closed under the following operations:
\begin{itemize}
\item If $\Gamma,\Gamma'\in\mathcal{G}$, then $\Gamma\sqcup\Gamma'\in \mathcal{G}$,
\item if $\Gamma\in \mathcal{G}$, then $\Gamma\ast\{v\}\in\mathcal{G}$.
\end{itemize}
Here $\sqcup$ denotes disjoint union and $\Gamma\ast\{v\}$ is the join (defined below) of $\Gamma$ and $\{v\}$.
This section will be devoted to proving the following theorem:

\begin{theorem}
\label{Thm:RAAG}
Let $\Gamma$ be a graph and $A(\Gamma)$ the associated \mbox{RAAG}. If $A(\Gamma)$ has multiple context-free word problem, then $\Gamma\in \mathcal{G}$.
\end{theorem}

This theorem would have a much cleaner statement if one could prove the following conjecture:
\begin{conjecture}\label{f2zconj}
The word problem for $F_2\times \ZZ$ is not MCF. 
\end{conjecture}

This would prove (and by work of \cite{kropholler_freeprod_2017} is equivalent to the following):

\begin{conjecture}
A RAAG $A(\Gamma)$ has MCF word problem if and only if $\Gamma$ is a disjoint union of cliques.
\end{conjecture}

\subsection{Graph theory and RAAGs}

Right-angled Artin groups (RAAG's) have been the subject of much recent interest because of their rich subgroup structure; in particular every special group embeds in a RAAG. See~ \cite{MR3104553, wise1, wise2}. 

\begin{definition}
Let $\Gamma$ be a graph (more precisely, an undirected graph with no loops). The associated {\em right angled Artin group} $A(\Gamma)$ is the group with presentation:
$$\langle v\in V(\Gamma)\mid [v, w]\mbox{ if }[v,w]\in E(\Gamma)\rangle$$
\end{definition}

\begin{definition}
\begin{enumerate}
\item 
$K_1$ is the graph with one vertex and no edges. 
\item $P_4$ is the graph with 4 vertices and 3 edges depicted in Figure~\ref{fig:len3}.
\end{enumerate}
\end{definition}

\begin{figure}[ht]
\begin{tikzpicture}[-,>=stealth',auto,node distance=1.5cm,
                    semithick]
  \node[state]            (A)               {$a$};
  \node[state]            (B) [right of=A]  {$b$};
  \node[state]            (C) [right of=B]  {$c$};
  \node[state]            (D) [right of=C]  {$d$};
  \path (A) edge              (B)
        (B) edge              (C)
        (C) edge              (D);
\end{tikzpicture}
\caption{The graph $P_4$ \label{fig:len3}} 
\end{figure}
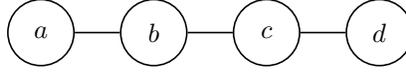

\begin{definition}\label{Def:Join}
A graph $\Gamma$ is a {\em join} if there exist non-empty induced subgraphs $J, K\subset\Gamma$ such that the following hold:
\begin{itemize}
\item $V(\Gamma) = V(J)\sqcup K(L)$,
\item every vertex of $J$ is joined to every vertex of $K$. 
\end{itemize}
We write $\Gamma = J\ast K$ if $\Gamma$ is a join of $J$ and $K$. 
\end{definition}

Clearly $A(\Gamma)=A(J)\times A(K)$ if  $\Gamma = J\ast K$. It follows from Servatius' Centralizer Theorem~\cite{MR1023285} that $A(\Gamma)$ is a non-trivial direct product if and only if $\Gamma$ is a join.
For example $A(P_4)$ is not a direct product.

There is a nice characterisation of joins using complement graphs.

\begin{definition}
Let $\Gamma$ be a graph. Its \emph{complement} $\bar\Gamma$ is defined as follows:
\begin{itemize}
\item $V(\bar\Gamma) = V(\Gamma)$,
\item two vertices $v, w$ are joined by an edge in $\bar\Gamma$ if and only if they are not joined by an edge in $\Gamma$. 
\end{itemize}
\end{definition}

\begin{remark}
\label{Rem:Dual}
Complementation is an involution on the set of graphs (that is, $\bar{\bar\Gamma} = \Gamma$). Notice that $P_4$ is isomorphic to its own complement.
\end{remark}

\begin{lemma}
A graph $\Gamma$ is a join if and only if $\bar\Gamma$ is disconnected. 
\end{lemma}
\begin{proof}
Suppose $\Gamma=J*K$.  Then in $\Gamma^*$ there are no edges from any vertex of $J$ to any vertex of $K$. For the converse, use \refrem{Dual}.
\end{proof}

Complements respect induced subgraphs as follows. 

\begin{lemma}
Let $\Gamma$ be a graph. If $\Lambda\subset\Gamma$ is a full subgraph, then $\bar\Lambda \subset \bar\Gamma$ is a full subgraph. \qed
\end{lemma}

\begin{definition} The class, CoG, of \emph{complement reducible} graphs is the smallest clase which contains $K_1$ and is closed under complement and disjoint union. For short we speak of {cographs} instead of complement reducible graphs. 
\end{definition}

\begin{theorem}[\cite{MR1686154}]\label{Thm:CoG}
\begin{enumerate}
\item A connected cograph is either a join or the graph with a single vertex.
\item A graph is a cograph if and only if it has no full $P_4$ subgraphs.
\end{enumerate}
\end{theorem}

\subsection{Proof of~\refthm{RAAG}}

\begin{theorem}\label{Thm:nol3}
The word problem for $A(P_4)$ is not MCF. 
\end{theorem} 
\begin{proof}
Recall $A(P_4) = \langle a, b, c, d\mid [a, b], [b, c], [c, d]\rangle$. Let $W$ denote the word problem in $A(P_4)$. We will consider the Bestvina-Brady group $BB(P_4)$, which is the kernel of the following homomorphism:
\begin{align*}
A(P_4) &\to \ZZ\\
a&\mapsto 1\\
b&\mapsto 1\\
c&\mapsto 1\\
d&\mapsto 1.
\end{align*}
By \cite{dicks_presentations_1999}, $BB(P_4)$ is a free group of rank three generated by $\{x = ab^{-1}, y = bc^{-1}, z = cd^{-1}\}$. We will study the language $L = W\cap R$, where $R$ denotes the regular language $(ad)^*(a^{-1}d^{-1})^* \{x, y, z\}^*$. By counting exponents we see that $$L \subset \{(ad)^n(a^{-1}d^{-1})^n \{x, y, z\}^*\}.$$

Let $$u_n = xy^{2n-1}z^{-1}$$ 
and 
$$v_n = x^{-1}y^{2n-1}z.$$ 
Note that in the group $A(P_4)$ we have equalities 
$$u_n = b^{2n-2}(ad)c^{-2n}$$ 
and 
$$v_n = b^{2n}(a^{-1}d^{-1})c^{2-2n}.$$
We can thus see that 
$$(ad)^nc^{-2n} = u_1y^{-2}u_2y^{-4}\dots y^{2-2n}u_n$$ 
and 
$$b^{2n}(a^{-1}b^{-1})^n = v_ny^{2-2n}v_{n-1}\dots y^{-2}v_1.$$ 
Combining these, we have 
$$(ad)^n(a^{-1}d^{-1})^n = u_1y^{-2}u_2y^{-4}\dots u_ny^{-2n}v_n\dots y^{-2}v_1.$$
Since $BB(P_4)$ is a free group this is a minimal representation of this element. Thus the positive exponent sum of $y$ in any word representing $(ad)^n(a^{-1}d^{-1})^n$ is greater than or equal to $2n^2$. We can now consider the image of the Parikh map:
\begin{align*}
L&\to \NN^2\\
w&\mapsto (|w|_a, |w|_y).
\end{align*}
The image of this lies on and above the curve $y = 2n^2$, thus any non-vertical line intersects this set in a finite subset. Hence $L$ is not semilinear and neither is $W$. We conclude, by Lemma \ref{Lem:Useful}, that the word problem in $A(P_4)$ is not MCF.
\end{proof}

\begin{theorem}\label{Thm:nof2f2}
The word problem for $F_2\times F_2$ is not MCF. 
\end{theorem}
\begin{proof}
Let $F_2$ be free on $\{a,b\}$, and let $f\colon F_2\to Z^2$ be the abelianisation
map $F_2\to \ZZ^2$. The fibre product of $f$ is $P = \{(u, v)\in F_2\times F_2 \mid f(u) = f(v)\}$. It is easy to show that $P$ is generated by $r = (a, a), s = (b,b), t =
(aba^{-1}b^{-1}, 1)$. By~\cite[Theorem 2]{olshanskii_length_2001}, $P$ is quadratically distorted in
$F_2\times F_2$. In particular, any word in $r, s$ and $t$
representing the element $(a^nb^ma^{-n}b^{-m}, 1)$ has at least $nm$ occurrences of $t$.

Consider the intersection of the word problem $W$ with the regular language $R$ 
\[ L = W\cap R = W\cap a^*b^*(a^{-1})^*(b^{-1})^*\{r, s, t, r^{-1}, s^{-1}\}^*.
\] 

Look at the image of $L$ under the Parikh map:
\begin{align*}
L &\to \NN^2\\
w &\mapsto (|w|_a, |w|_t).
\end{align*}
The image of this map is $\{(n, nm)\}$ and by Lemma~\ref{Lem:iij} this is not a semilinear set. Hence, $L$ is not semilinear and therefore, not MCF. It follows, by Lemma \ref{Lem:Useful}, that the word problem in $F_2\times F_2$ is not MCF.
\end{proof}

\begin{proof}[Proof of \refthm{RAAG}]
By~\cite{kropholler_freeprod_2017}, the class of groups with MCF word problem is closed under free products. We can therefore reduce to connected graphs $\Gamma$. The class of groups with MCF word problem is closed under taking finitely generated subgroups. We will now consider connected graphs $\Gamma$ and the associated RAAG $A(\Gamma)$. By Theorems \ref{Thm:nol3} and \ref{Thm:nof2f2}, the graph $\Gamma$ cannot contain any full subgraphs isomorphic to $P_4$ or a square. 

By Theorem~\ref{Thm:CoG} a connected graph which does not contain an induced subgraph $P_4$ is the join of two induced subgraphs $J$ and $K$. As $J$ and $K$ are induced subgraphs, they also contain no copies of $P_4$. Thus if connected they split as a join and so on. 

Repeating this splitting process we see $\Gamma = A_0\ast A_1\ast\dots\ast A_n$. If $\mbox{Diam}(A_i)>1$ for more than one $i$, then the graph contains a square. By maximality of the splitting, we can assume that $A_i = \{v\}$ for all $i\neq 0$. If $A_0$ is connected, then, by maximality of the spltting, it is a point and $A(\Gamma) = \ZZ^n$. In the case that $A_0$ is disconnected, we can use the above analysis to decompose the connected components of $A_0$. Repeating this process we see that $\Gamma\in\mathcal{G}$. 
\end{proof}

\bibliography{GKS}{}
\bibliographystyle{hyperplain} % replaces plain.bst

\end{document}